\documentclass[a4paper,twoside,10pt]{article}

\usepackage[a4paper,left=3cm,right=3cm, top=3cm, bottom=3cm]{geometry}
\usepackage[latin1]{inputenc}
\usepackage{mathrsfs}
\usepackage{graphicx}
\usepackage{epstopdf}
\usepackage{amsmath}
\usepackage{appendix}
\usepackage{amsthm}
\usepackage{amssymb}
\usepackage{hyperref}
\usepackage{stmaryrd}
\usepackage{float}
\usepackage{bigints}
\usepackage{cite}
\usepackage{color}
\usepackage[abs]{overpic}
\usepackage[font=footnotesize,labelfont=bf]{caption}
\usepackage{cases}
\usepackage{tikz}
\usepackage{algorithm,algpseudocode}
\usepackage{rotating}
\usepackage{blkarray}
\usetikzlibrary{matrix,calc,arrows}
\usepackage{soul,xcolor}
\usepackage{verbatim}
\usepackage{graphicx}
\usepackage{subcaption}
\usepackage{algorithm}
\usepackage{pifont}
\usepackage{multirow}

\restylefloat{table}
\theoremstyle{plain}
\newtheorem{thm}{Theorem}[section]
\newtheorem{cor}[thm]{Corollary}
\newtheorem{lem}[thm]{Lemma}
\newtheorem{prop}[thm]{Proposition}
\theoremstyle{definition}

\theoremstyle{remark}
\newtheorem{remark}{Remark}

\definecolor{darkseagreen}{rgb}{0.56, 0.74, 0.56}
\definecolor{darkspringgreen}{rgb}{0.09, 0.45, 0.27}

\newcommand{\Vertiii}[1]{{\left\vert\kern-0.25ex\left\vert\kern-0.25ex\left\vert #1 \right\vert\kern-0.25ex\right\vert\kern-0.25ex\right\vert}}

\newcommand{\E}{K}
\newcommand{\F}{F}
\newcommand{\e}{e}
\newcommand{\taun}{\mathcal T_n}
\newcommand{\Fcaln}{\mathcal F_n}
\newcommand{\Fcaltilden}{\widetilde {\mathcal F}_n}
\newcommand{\FcaltildenF}{\Fcaltilden(\F)}
\newcommand{\Ecaln}{\mathcal E_n}
\newcommand{\tautilden}{\widetilde{\mathcal T}_n}
\newcommand{\tautildenE}{\tautilden(\E)}
\newcommand{\h}{h}
\newcommand{\hE}{\h_\E}
\newcommand{\hF}{\h_\F}
\newcommand{\he}{\h_\e}
\newcommand{\Vzface}{\mathbf V_0^{\text{face}}}
\newcommand{\VzfaceE}{\Vzface(\E)}
\newcommand{\VzfaceF}{\Vzface(\F)}
\newcommand{\Vzedge}{\mathbf V_0^\text{edge}}
\newcommand{\VzedgeF}{\Vzedge(\F)}
\newcommand{\VzedgeE}{\Vzedge(\E)}
\newcommand{\psibold}{\boldsymbol \psi}
\newcommand{\phibold}{\boldsymbol \phi}
\newcommand{\psiboldh}{\psibold_\h}
\newcommand{\phiboldh}{\phibold_\h}
\newcommand{\ubold}{\mathbf u}
\newcommand{\vbold}{\mathbf v}
\newcommand{\vboldF}{\mathbf v ^\F}
\newcommand{\vboldhF}{\mathbf {v}_\h ^\F}
\newcommand{\vboldhpartialE}{\mathbf {v}_\h ^{\partial \E}}
\newcommand{\vboldpartialE}{\mathbf v ^{\partial \E}}
\newcommand{\vboldpartialEh}{\mathbf v ^{\partial \E}_\h}
\newcommand{\w}{w}
\newcommand{\wbold}{\mathbf w}

\newcommand{\zbold}{\mathbf z}
\newcommand{\uboldh}{\ubold_\h}
\newcommand{\vboldh}{\vbold_\h}
\newcommand{\wboldh}{\wbold_\h}

\let\div\relax
\DeclareMathOperator{\div}{div}
\DeclareMathOperator{\divF}{div_\F}
\DeclareMathOperator{\curl}{curl}
\DeclareMathOperator{\curlbold}{\bf{curl}}
\DeclareMathOperator{\curlboldF}{{\bf{curl}}_\F}

\DeclareMathOperator{\rot}{rot}
\DeclareMathOperator{\rotF}{rot_\F}
\DeclareMathOperator{\DeltaF}{\Delta_\F}
\DeclareMathOperator{\nablaF}{\nabla_\F}
\newcommand{\Pbb}{\mathbb P}
\newcommand{\Rbb}{\mathbb R}
\newcommand{\Nbb}{\mathbb N}
\newcommand{\n}{\mathbf n}
\newcommand{\nE}{\n_\E}
\newcommand{\nF}{\n_\F}
\renewcommand{\ne}{\n_\e}
\renewcommand{\t}{\boldsymbol \tau}

\newcommand{\te}{\t_\e}
\newcommand{\tpartialF}{{\t_{\partial \F}}}
\newcommand{\EE}{\mathcal F^\E}
\newcommand{\EF}{\mathcal E^\F}
\newcommand{\xbold}{\mathbf x}
\newcommand{\xboldE}{\xbold_\E}
\newcommand{\xboldF}{\xbold_\F}
\newcommand{\xE}{x_\E}
\newcommand{\yE}{y_\E}
\newcommand{\zE}{z_\E}
\newcommand{\xF}{x_\F}
\newcommand{\yF}{y_\F}
\newcommand{\rhobold}{\boldsymbol \rho}
\newcommand{\qbold}{\mathbf q}
\newcommand{\qboldo}{\qbold_1}
\newcommand{\bE}{b_\E}
\newcommand{\bpartialE}{b_{\partial \E}}
\newcommand{\bF}{b_\F}
\newcommand{\psiboldpi}{\psibold_\pi}
\newcommand{\psiboldI}{\psibold_I}

\newcommand{\vboldpi}{\vbold_\pi}
\newcommand{\vboldI}{\vbold_I}

\newcommand{\Pibold}{\boldsymbol \Pi}
\newcommand{\Piboldz}{\Pibold_0}
\newcommand{\Piboldface}{\Pibold_{\text{face}}^0}
\newcommand{\Piboldedge}{\Pibold_{\text{edge}}^0}
\newcommand{\SE}{S^\E}
\newcommand{\SF}{S^\F}
\DeclareMathOperator{\FACE}{face}
\DeclareMathOperator{\EDGE}{edge}
\DeclareMathOperator{\SEface}{\SE_{\FACE}}
\DeclareMathOperator{\SEedge}{\SE_{\EDGE}}
\DeclareMathOperator{\SFface}{\SF_{\FACE}}
\DeclareMathOperator{\SFedge}{\SF_{\EDGE}}
\newcommand{\Ibold}{\mathbf I}
\newcommand{\Phibold}{\boldsymbol \Phi}
\newcommand{\npartialF}{{\n_{\partial \F}}}
\newcommand{\That}{\widehat T}
\newcommand{\what}{\widehat w}
\newcommand{\vboldhat}{\widehat {\vbold}}
\newcommand{\thatehat}{\widehat {\boldsymbol \tau}_{\ehat}}

\newcommand{\ehat}{\widehat e}
\newcommand{\Fhat}{\widehat \F}
\DeclareMathOperator{\rothatFhat}{\widehat{\rot}_{\Fhat}}

\begin{tiny}
\author{
\normalsize{
}}
\end{tiny}

\date{}

\title{{\textbf{\large{Interpolation and stability properties of low order face and edge virtual element spaces}}}}
\date{}
\author{{L. Beir\~ao da Veiga \thanks{Dip. di Matematica e Applicazioni,  Universit\`a degli Studi di Milano-Bicocca, Italy (lourenco.beirao@unimib.it)},\quad
L. Mascotto\thanks{Dip. di Matematica e Applicazioni,  Universit\`a degli Studi di Milano-Bicocca, Italy (lorenzo.mascotto@unimib.it),
Fakult\"at f\"ur Mathematik, Universit\"at Wien, 1090 Vienna, Austria (lorenzo.mascotto@univie.ac.at)}}}

\begin{document}
\maketitle
\begin{abstract}
\noindent
We analyse the interpolation properties of 2D and 3D low order virtual element face and edge spaces, which generalize N\'ed\'elec and Raviart-Thomas polynomials to polygonal-polyhedral meshes.
Moreover, we investigate the stability properties of the associated $L^2$-discrete bilinear forms, which typically appear in the virtual element discretization of problems in electromagnetism.

\medskip


\end{abstract}

\section{Introduction}  \label{section:introduction}

Conforming $H(\div)$ and $H(\curlbold)$ elements \cite{RaviartThomas1977,nedelec1986,BBF-book} are fundamental in the discretization of a variety of problems, such as diffusion problems in mixed form and electromagnetic equations.
Virtual element (VE) face and edge spaces generalize N\'ed\'elec and Raviart-Thomas polynomials to polygonal-polyhedral meshes.
When coupled with nodal virtual element spaces and piecewise polynomials for the discretization of~$H^1$ and~$L^2$, they constitute a discrete exact complex.
Such VE spaces were first introduced in~\cite{HdivHcurlVEM} and later improved in a series of papers dealing with magnetostatic problems:
\cite{da2017virtual} for the two dimensional case; \cite{da2018lowest, da2018family} for the three dimensional case.

The definition of VE spaces in the above papers goes hand in hand with the introduction of discrete scalar products mimicking the $L^2$ scalar product used to discretize the problem of interest.
Such scalar products and in general discrete bilinear forms are constructed upon two main ingredients: a projection that maps virtual element functions into a polynomial (sub)space and a computable ``stabilization''.

For standard $H^1$-conforming virtual element spaces~\cite{VEMvolley}, there is a vast literature dealing with the interpolation properties of the spaces and the stability properties of the associated discrete bilinear forms;
see, e.g., \cite{VEMchileans, cangianigeorgulispryersutton_VEMaposteriori, brennerVEMsmall, BrennerGuanSung_someestimatesVEM, beiraolovadinarusso_stabilityVEM, chen_anisotropic_conforming}.
Notwithstanding, only limited results regarding the approximation properties of face and edge virtual elements are available.
In~\cite{VEM:acoustic-vibration}, face elements in 2D are analyzed in a non-enhanced setting, simpler than that in~\cite{da2017virtual}, and no result on the stability properties of the discrete scalar products are provided.

The goal of this contribution is to rigorously present (\emph{i}) the interpolation properties of 3D face and edge VE spaces and (\emph{ii}) the stability properties of the associated discrete scalar products.
In order to better underline the challenges and the interesting aspects of the analysis, we focus on the lowest order case introduced in~\cite{da2018lowest}.
Although the extension to the higher order case may hide additional technical difficulties, the path to follow is essentially the same drawn in this paper.
As outlined in the Appendix, the 2D analysis is simpler and follows along the same lines as the 3D one.
The main challenges we have to cope with are related to the virtual nature of the involved spaces, which contain polynomials of a suitable maximum degree, but also other functions that are defined through a PDE problem on the element.
Furthermore, the lack of a reference element, which is an important tool in the FEM analysis, introduces several complications, such as rendering even standard inverse estimates in VE spaces quite challenging.
We indeed believe that some proving techniques introduced in the present article can turn out to have a general value, representing useful tools and ideas for the VEM community. 
For instance, in order to derive the stability of the interpolation operator, we need to combine a Helmholtz decomposition with additional nonstandard manipulations in order to handle the serendipity constraint,
and express the ensuing bound in terms of evaluations that represent the degrees of freedom.
The stability of the scalar product hinges on a critical result that requires to split and re-assemble the error with an ad-hoc argument that we find quite interesting.


The paper is organized as follows.
In Section~\ref{section:preliminaries}, we introduce regular meshes and the related notation. Moreover, we recall a number of technical results in Sobolev spaces.
We address the interpolation properties in 3D face and edge VE spaces in Sections~\ref{section:faceVEM-loworder} and~\ref{section:edge3DVEM-loworder}, respectively.
We analyze the properties of explicit stabilizations for face and edge VE spaces in Section~\ref{section:stabilizations}.
Appendix~\ref{appendix:interpolation-properties-2D} deals with the 2D counterpart of the results presented in the paper.

We conclude this introduction setting the necessary notation and functional setting.
\paragraph*{Notation and functional spaces.}
Given two positive quantities~$a$ and~$b$, we use the short-hand notation ``$a \lesssim b$'' if there exists a positive constant~$c$ independent of the discretization parameters such that ``$a \le c \, b$''.
Moreover, we write ``$a \approx b$'' if and only if ``$a \lesssim b$ and $b\lesssim a$''.

We recall the definition of some differential operators.
Given a face~$\F$ parallel to the $xy$-plane and~$\vbold=(v_1,v_2) : \F \subseteq \Rbb^2 \rightarrow \Rbb^2$ and~$v : \F \subseteq \Rbb^2 \rightarrow \Rbb$, define
\[
\begin{split}
& \divF \vbold: = \partial_x v_1 + \partial_y v_2, \quad\quad \rotF \vbold := -\partial_y v_1 + \partial_x  v_2, \quad \quad \curlboldF v  := (\partial_y v, -\partial _x v)^T.
\end{split}
\]
The above definitions trivially extend to the case of all faces~$\F$.

Furthermore, given~$\vbold = (v_1, v_2, v_3) : \E \subseteq \Rbb^3 \rightarrow \Rbb^3$, we define
\[
\div \vbold := \partial_x v_1 + \partial_y v_2 + \partial_z v_3, \quad \quad \curlbold \vbold := (\partial_y v_3 - \partial_z v_2, \partial _z v_1 - \partial_x v_3, \partial _x v_2 - \partial_y v_1)^T.
\]
Next, we introduce several Sobolev spaces. Given a domain~$D \subset \Rbb^d$, $d=2,3$, we begin by defining the Sobolev space of order~$s>0$ by~$H^s(D)$, which we endow with the norm and seminorm~$\Vert \cdot \Vert_{s,D}$ and~$\vert \cdot \vert_{s,D}$.
The case~$s=0$ boils down to the Lebesgue space~$L^2(D)$. Negative Sobolev spaces are defined via duality.
In particular, for~$\vbold : D\subset \Rbb^{d_1} \to \Rbb ^{d_2}$, $d_1$, $d_2\in \Nbb$, we write
\[
\Vert \vbold \Vert_{-1,D} := \sup_{\Phibold \in [H^1_0(D)]^{d_2}} \frac{(\vbold, \Phibold)_{0,D}}{\vert \Phibold \vert_{1,D}}.
\]
We introduce~$H(\div,D)$ as the space of~$L^2(D)$ functions having weak divergence in~$L^2(D)$.
Analogously, in the~$d=2$ and~$d=3$ cases, we introduce~$H(\rotF,D)$, $H(\curlbold,D)$, and~$H(\curlbold\, \curlbold,D)$ as the spaces of~$[L^2(D)]^2$ and~$[L^2(D)]^3$ functions having weak rotor, $\curlbold$, and~$\curlbold\, \curlbold$ in~$L^2(D)$, $[L^2(D)]^3$, and~$[L^2(D)]^3$, respectively.
Finally, for~$D \subset \Rbb^3$ and~$s>0$, we define~$H^s(\curlbold,D)$ as the subspace of functions~$\vbold \in [H^s(D)]^3$ such that~$\curlbold \vbold \in [H^s(D)]^3$.

\section{Meshes and preliminaries}  \label{section:preliminaries}
Henceforth, let~$\Omega \subset \Rbb^3$ denote a given polyhedral domain and~$\{ \taun \}_{n\in \Nbb}$ be a sequence of polyhedral meshes.
For all~$n \in \Nbb$, we denote the set of faces and edges by~$\Fcaln$ and~$\Ecaln$ respectively.
Moreover, given~$\E \in \taun$, its set of faces is~$\EE$, whereas, given a face~$\F \in \Fcaln$, its set of edges is~$\EF$.
The diameter of an element~$\E \in \taun$ and a face~$\F \in \Fcaln$ are~$\hE$ and~$\hF$, whereas the length of an edge~$\e \in \Ecaln$ is~$\he$.

We denote the outward unit normal vector to the boundary of an element~$\E$ by~$\nE$.
For each face~$\F$ in~$\partial \E$, we fix~$\nF$ as~${\nE}_{|_\F}$.
For future convenience, for every face~$\F$, we also introduce~$\npartialF$ defined as the outward unit normal vector to~$\partial \F$ in the plane containing~$\F$.
Then, for each edge~$\e$ in~$\partial \F$, we fix~$\ne$ as~$\npartialF{}_{|\e}$ in the plane containing~$\F$. 
Further, given an edge~$\e$ of $\partial F$, $\te$ represents the tangent unit vector of~$\e$ pointed in counter-clockwise sense with respect to the couple $(\F,\nF)$.
We introduce~$\tpartialF$ such that~$\tpartialF{}_{|\e} = \te$ for every edge~$\e \in \EF$.

In what follows, $(\xE,\yE,\zE)$ is the barycenter of an element~$\E$, whereas~$(\xF,\yF)$ represents the barycenter of a face~$\F$.

\medskip

Henceforth, we demand that~$\{ \taun \}_{n\in\Nbb}$ satisfies the following assumptions: for all~$n \in \Nbb$,
\begin{itemize}
\item every element~$\E \in \taun$ and face~$\F \in \Fcaln$ 
is star-shaped with respect to a ball of diameter~$\gamma \hE$ and~$\gamma \hF$, respectively;
\item given an element~$\E \in \taun$, for all its faces~$\F \in \EE$, $\hE \approx \hF$; moreover, given a face~$\F \in \Fcaln$, for all its edges~$\e \in \EF$, $\hF \approx \he$;
\item every element~$\E \in \taun$ is Lipschitz.
\end{itemize}

The above couple of assumptions is the accepted standard in the virtual element literature (taking the role of the classical shape-regularity for tetrahedral meshes in finite elements).
Such conditions have been relaxed only in a few papers, see, e.g., \cite{beiraolovadinarusso_stabilityVEM, brennerVEMsmall, chen_anisotropic_conforming}, all dealing with simpler $H^1$-elliptic scalar problems. Note that, as a simple consequence of the conditions above, the number of edges per face and the number of faces per element are uniformly bounded in the mesh family.

We state here some technical results that we shall employ in the forthcoming analysis.
The first one deals with several trace type inequalities.
\begin{lem} \label{lemma:Agmon}
Given~$\E \in \taun$ with size~$\hE$, the following inequality is valid: for all~$v \in H^s(\E)$, $s>1/2$,
\begin{equation} \label{Agmon}
\Vert v \Vert_{0,\partial \E} \lesssim \hE^{-\frac{1}{2}}\Vert v \Vert_{0,\E} +  \hE^{s-\frac{1}{2}} \vert v \vert_{s, \E}.
\end{equation}
Moreover, if additionally~$1/2 < s \le 1$ and~$v$ has zero average on either~$\partial \E$ or~$\E$, then the following bound is valid as well:
\begin{equation} \label{Agmon:zero-average}
\Vert v \Vert_{0,\partial \E} \lesssim \hE^{s-\frac{1}{2}} \vert v \vert_{s, \E}.
\end{equation}
Finally, if~$v \in H^1(\E)$ and has zero average on either~$\partial \E$ or~$\E$, we also have
\begin{equation} \label{multiplicative:trace}
\Vert v \Vert_{0,\partial \E} \lesssim  \Vert v \Vert_{0,\E}^{\frac{1}{2}} \vert v \vert_{1,\E}^{\frac{1}{2}}.
\end{equation}
\end{lem}
\begin{proof}
Bound~\eqref{multiplicative:trace} is the standard multiplicative trace; see, e.g., in~\cite[Section~1.6]{BrennerScott}.
Bound~\eqref{Agmon} is a standard (scaled) trace inequality,  whereas we obtain~\eqref{Agmon:zero-average} by using~\eqref{Agmon} and the Poincar\'e inequality.
\end{proof}

Next, we recall the Friedrichs' inequality.
\begin{lem} \label{lemma:Friedrichs-Monk}
Let~$\E \in \taun$ and~$\vbold$ be a divergence free function satisfying~$\vbold{}_{|\partial \E} \times \nE \in [L^2(\partial \E)]^2$. Then, the following bound is valid:
\begin{equation} \label{Friedrichs-Monk}
\Vert \vbold \Vert_{0, \E} \lesssim \hE \Vert \curlbold \vbold \Vert_{0,\E} + \hE^{\frac{1}{2}} \Vert \vbold \times \nE \Vert_{0,\partial \E}.
\end{equation}
\end{lem}
\begin{proof}
This bound is a scaled version of~\cite[Corollary~3.51]{monk2003finite}.
\end{proof}

Further, we recall two polynomial inverse inequalities in polyhedral domains.
\begin{lem} \label{lemma:polynomial-inverse}
Given~$\E \in \taun$, the following polynomial inverse estimate is valid: for all scalar and vector polynomials~$\qbold$ of bounded degree in~$\E$,
\begin{equation} \label{H1-L2}
\vert  \qbold  \vert_{1,\E} \lesssim \hE^{-1} \Vert \qbold \Vert_{0,\E} .
\end{equation}
Set the scaled norm
\begin{equation} \label{H-1/2:norm}
 \Vert \cdot \Vert_{-\frac{1}{2}, \partial \E} := \sup_{\varphi \in H^{\frac12}(\partial \E)} \frac{(\varphi, \cdot)_{0,\partial \E}}{\vert \varphi \vert_{\frac{1}{2},\partial \E} + \hE^{-\frac{1}{2}} \Vert \varphi \Vert_{0,\partial \E}}.
\end{equation}
Then, for every piecewise, i.e., face by face scalar polynomials~$q$ of bounded degree over~$\partial \E$, we also have
\begin{equation} \label{L2-H1/2}
\Vert q \Vert_{0,\partial \E}  \lesssim \hE ^{-\frac{1}{2}}  \Vert q \Vert_{-\frac{1}{2}, \partial \E},
\end{equation}
The estimates above are valid substituting the element~$\E$ with a face~$\F$ as well.
\end{lem}
\begin{proof}
Thanks to the regularity assumptions on the mesh in Section~\ref{section:preliminaries}, 
each element~$\E \in \taun$ can be split into a sub-tessellation~$\tautildenE$ consisting of a uniformly bounded set of shape-regular tetrahedra.
Clearly, $\tautilden(\E)$ induces a shape-regular sub-triangulation~$\FcaltildenF$ on each face~$\F$ of~$\E$.

As for~\eqref{H1-L2}, it suffices to apply a standard polynomial inverse estimate over each tetrahedron of~$\tautildenE$;
see, e.g., \cite[Proposition~3.37]{verfurth2013posteriori} and the references therein.

As for~\eqref{L2-H1/2}, on each face~$\F \in \EE$, we introduce~$\bF$, which is defined as the piecewise cubic bubble function over~$\FcaltildenF$ such that~$\Vert \bF \Vert_{\infty, T} = 1$ for all~$T$ in~$\FcaltildenF$.
Also define a function~$\bpartialE$ over~$\partial \E$ satisfying~$\bpartialE{}_{|\F}=\bF$.
A standard bubble argument triangle by triangle, recalling that $q$ is a polynomial for all faces~$\F \in \EE$, yields
$$
\Vert \bF\, q \Vert_{0,\F}  \approx \Vert \bF^{1/2} \, q \Vert_{0,\F}  \approx \Vert q \Vert_{0,\F} .
$$
For the same reason, a polynomial inverse estimate gives
\[
\vert \bF \, q \vert_{1,\F} \lesssim \hE^{-1} \Vert \bF \, q \Vert_{0,\F} .
\]
The above inequalities trivially imply
\[
\Vert \bpartialE\, q \Vert_{0,\partial\E}  \approx 
\Vert \bpartialE^{1/2} \, q \Vert_{0,\partial\E} \approx \Vert q \Vert_{0,\partial\E} \ , \qquad
\vert \bpartialE \, q \vert_{1,\partial\E} \lesssim \hE^{-1} \Vert \bpartialE \, q \Vert_{0,\partial\E}.
\]
In turn, interpolation theory~\cite{Triebel} yields
\begin{equation} \label{second-step-inverse1D}
\vert \bpartialE \, q \vert_{\frac{1}{2}, \partial\E} \lesssim \hE^{-\frac12} \Vert \bpartialE \, q \Vert_{0,\partial\E}.
\end{equation}
With this at hand, we can write
\[
\Vert q \Vert_{-\frac{1}{2}, \partial \E} 
= \sup_{\varphi \in H^{\frac{1}{2}}(\partial \E)} \frac{(\varphi,  q)_{0, \partial \E}}{\vert \varphi \vert_{\frac{1}{2}, \partial \E} + \hE^{-\frac12} \Vert \varphi \Vert_{0, \partial \E}  }
\ge \frac{\Vert \bpartialE^{\frac{1}{2}} q \Vert^2_{0, \partial \E}}{\vert \bpartialE \, q \vert_{\frac{1}{2}, \partial \E} + \hE^{-\frac12} \Vert \bpartialE\, q \Vert_{0, \partial \E}} 
\overset{\eqref{second-step-inverse1D}}{\gtrsim} \hE^{\frac{1}{2}} \Vert q \Vert_{0, \partial \E},
\]
which is the assertion.
\end{proof}

Finally, we recall the $\div$- and $\curlbold$-trace inequalities.

\begin{lem} \label{lemma:curl-div-trace}
Given~$\E \in \taun$, let~$\psibold \in H(\div, \E)$ and~$\vbold \in H(\curlbold, \E)$. Then, the two following trace inequalities are valid:
\begin{equation} \label{div-trace}
\Vert \psibold \cdot \nE \Vert_{-\frac{1}{2}, \partial \E} \lesssim \Vert \psibold \Vert_{0,\E} + \hE \Vert \div \psibold \Vert_{0,\E}
\end{equation}
and
\begin{equation} \label{curl-trace}
\Vert \vbold \times \nE \Vert_{-\frac{1}{2}, \partial \E} \lesssim \Vert \vbold \Vert_{0,\E} + \hE \Vert \curlbold \vbold \Vert_{0,\E}.
\end{equation}
Moreover, given a face~$\F$, let~$\vbold \in H(\rotF, \F)$ and ${\bf w} \in H(\divF, \F)$.
Then, the following trace inequalities are valid as well:
\begin{eqnarray} \label{rot-trace}
& \Vert \vbold \cdot \tpartialF \Vert_{-\frac{1}{2}, \partial \F} \lesssim \Vert \vbold \Vert_{0,\F} + \hF \Vert \rotF \vbold \Vert_{0,\F} , \\
\label{added:1}
& \Vert {\bf w} \cdot \npartialF \Vert_{-\frac{1}{2}, \partial \F} \lesssim \Vert {\bf w} \Vert_{0,\F} + \hF \Vert \divF {\bf w} \Vert_{0,\F}. 
\end{eqnarray}
\end{lem}
\begin{proof}
See, e.g., \cite[Section~3.5]{monk2003finite}.
\end{proof}

\begin{remark}
All the forthcoming interpolation estimates are local to the element and immediately yield the associated global results by summation on all mesh elements.
Analogously, we do not detail here the (conforming) global spaces associated to the local ones described here;
the definition of such global spaces, which  trivially follows by a standard FEM/VEM gluing procedure, can be found in~\cite{da2018lowest}.
\end{remark}

\section{Interpolation properties of 3D face spaces}  \label{section:faceVEM-loworder}

Let~$\E \in \taun$ be an element. Following~\cite{da2018lowest}, we define the low order local face virtual element space in 3D as follows:
\begin{equation} \label{definition:Vface}
\begin{split}
\VzfaceE := \{ \psiboldh 	& \in [L^2(\E)]^3 \mid  \div \psiboldh \in \Pbb_0(\E),\, \curlbold \psiboldh \in [\Pbb_0(\E)]^3,    \\ 
					& \psiboldh{}_{|\F} \cdot \nF \in \Pbb_0(\F) \, \forall \F \in \EE , \, \int_\E \psiboldh \cdot (\xboldE \times q_0)=0 \;\; \forall q_0 \in [\Pbb_0(\E)]^3  \},
\end{split}
\end{equation}
where we have set
\begin{equation} \label{def:xboldE}
\xboldE := (x-\xE, y - \yE, z - \zE)^T.
\end{equation}
In what follows, \emph{enhancing constraint} refers to the constraint
\begin{equation} \label{enhancing:face}
 \int_\E \psiboldh \cdot (\xboldE \times \qbold_0)=0 \quad \quad \forall \qbold_0 \in [\Pbb_0(\E)]^3.
\end{equation}
We endow the local space~$\VzfaceE$ with a set of unisolvent degrees of freedom, which is provided by the single value of~$\psiboldh \cdot \nF$ over each face~$\F \in \EE$.
In the following lemma, we show that such normal component values uniformly control the $L^2$ norm 
of functions in~$\VzfaceE$. 
\begin{prop} \label{prop:splitting:VEM-face}
The following a priori bound is valid for any~$\psiboldh \in \VzfaceE$:
\begin{equation} \label{a-priori:bound:psiboldh}
\Vert \psiboldh  \Vert_{0,\E} \lesssim \hE^{\frac{1}{2}}  \Vert \psiboldh \cdot \nE \Vert_{0,\partial \E}.
\end{equation}
\end{prop}
\begin{proof}
There exist~$\rhobold \in H(\curlbold\, \curlbold, \E) \cap H(\div,\E)$ and~$\Psi \in H^1(\E)\slash \Rbb$ such that the following Helmholtz decomposition is valid:
\begin{equation} \label{splitting:VEM-face}
\psiboldh =\curlbold \rhobold + \nabla \Psi.
\end{equation}
To prove~\eqref{splitting:VEM-face}, we first define a function~$\Psi \in H^1(\Omega) \slash \Rbb$ satisfying weakly
\begin{equation} \label{definition:Psi:face}
\begin{cases}
\Delta \Psi = \div \psiboldh 			& \text{in } \E\\
\nE \cdot \nabla \Psi	= \nE \cdot\psiboldh	& \text{on } \partial \E.
\end{cases}
\end{equation}
Next, we define a function~$\rhobold \in H(\curlbold,\Omega)$ satisfying weakly, see, e.g., \cite{Kikuchi1990},
\begin{equation} \label{definition:rho:face}
\begin{cases}
\curlbold\, \curlbold \rhobold = \curlbold \psiboldh 	& \text{in } \E \\
\div \rhobold= 0 								& \text{in } \E \\
\nE \times \rhobold  = 0 							& \text{on } \partial \E.
\end{cases}
\end{equation}
It can be readily checked that~$\rhobold$ and~$\Psi$ are defined so that~\eqref{splitting:VEM-face} is valid. Moreover, we also have the orthogonality condition $(\curlbold \rhobold, \nabla \Psi)_{0,\E} = 0$, whence we get
\begin{equation} \label{using:orthogonality:face}
\Vert \psiboldh  \Vert^2_{0,\E} =  \Vert \nabla \Psi \Vert ^2_{0,\E} + \Vert \curlbold \rhobold \Vert^2_{0,\E} .
\end{equation}
We bound the two terms on the right-hand side of~\eqref{using:orthogonality:face} separately. We begin with that involving the gradient.
Using an integration by parts, the definition~\eqref{definition:Psi:face}, the fact that~$\Psi$ has zero average over~$\E$, and~$\div \psiboldh \in \Pbb_0(\E)$, we get
\[
\Vert \nabla \Psi \Vert_{0,\E}^2 = - \int_\E \div \psiboldh\ \Psi + \int_{\partial \E} \nE \cdot \psiboldh \, \Psi =  \int_{\partial \E} \nE \cdot \psiboldh \, \Psi \le \Vert \psiboldh \cdot \nE \Vert_{0,\partial \E} \Vert \Psi \Vert_{0,\partial \E}.
\]
Further using inequality~\eqref{Agmon:zero-average} with $s=1$, we arrive at
\begin{equation} \label{bound:Psi-face}
\vert \Psi \vert_{1,\E} \lesssim \hE^{\frac{1}{2}}   \Vert \psiboldh \cdot \nE \Vert_{0,\partial \E}.
\end{equation}
Next, we deal with the second term on the right-hand side of~\eqref{using:orthogonality:face}. Integrating by parts and using~\eqref{definition:rho:face}, we obtain
\[
\Vert \curlbold \rhobold \Vert_{0,\E}^2 = \int_\E \curlbold \rhobold \cdot \curlbold \rhobold = \int_{\E} \rhobold \cdot \curlbold \,\curlbold \rhobold  = \int_\E \rhobold \cdot \curlbold \psiboldh .
\]
Set~$\qbold_0 := \curlbold \psiboldh$.
Since~$\qbold_0 \in [\Pbb_0(\E)]^3$, a direct calculation shows that
\[
\qbold_0 = \frac{1}{2} \curlbold (\qbold_0 \times \xboldE).
\]
Together with an integration by parts and the property~$\nE \times \rhobold {}_{|\partial \E} = 0$, see~\eqref{definition:rho:face}, this entails
\[
\begin{split}
\Vert \curlbold \rhobold \Vert_{0,\E}^2 = \frac{1}{2} \int_\E \rhobold \cdot \curlbold (\qbold_0 \times \xboldE) 	& = \frac{1}{2} \int_\E \curlbold \rhobold \cdot (\qbold_0 \times \xboldE) + \frac{1}{2} \int_{\partial \E} (\nE \times \rhobold) \cdot (\qbold_0 \times \xboldE) \\
																					& = \frac{1}{2} \int_\E \curlbold \rhobold \cdot (\qbold_0 \times \xboldE) .
\end{split}
\]
Next, we use~\eqref{splitting:VEM-face}, the enhancing constraint~\eqref{enhancing:face}, and a Cauchy-Schwarz inequality:
\[
\Vert \curlbold \rhobold \Vert_{0,\E}^2 = \frac{1}{2} \int_\E  (\psiboldh - \nabla \Psi) \cdot (\qbold_0 \times \xboldE)  =  - \frac{1}{2}\int_\E   \nabla \Psi \cdot (\qbold_0 \times \xboldE) \le \frac{1}{2} \vert \Psi \vert_{1,\E} \Vert  \qbold_0 \Vert _{0,\E} \Vert \xboldE \Vert_{\infty,\E}.
\]
Recalling that~$\qbold_0:=\curlbold \psiboldh = \curlbold \, \curlbold \rhobold$ and using~$\Vert \xboldE \Vert_{\infty, \E} \le \hE$, we end up with
\begin{equation} \label{bound-on-rhobold}
\Vert \curlbold \rhobold \Vert_{0,\E}^2  \lesssim \hE \vert \Psi \vert_{1,\E} \Vert \curlbold \curlbold \rhobold  \Vert_{0,\E}.
\end{equation}
We claim that the following inverse inequality is valid:
\begin{equation} \label{3D:inverse-curl}
\hE \Vert \curlbold \, \curlbold \rhobold \Vert_{0,\E} \lesssim \Vert \curlbold \rhobold \Vert_{0,\E}.
\end{equation}
To prove~\eqref{3D:inverse-curl}, define~$\bE$ as the piecewise quartic bubble function over the shape-regular decomposition of~$\E$ into tetrahedra~$T$; see the proof of Lemma~\ref{lemma:polynomial-inverse}.
More precisely, $\bE$ is the positive quartic polynomial over each tetrahedron~$T$ given by the product of the four barycentric coordinates of~$T$, scaled such that~$\Vert \bE \Vert_{\infty,T} = 1$.

Since~$\curlbold \, \curlbold \rhobold \in [\Pbb_0(\E)]^3$,
we can use the equivalence of norms in finite dimensional spaces on each tetrahedron~$T$ and an integration by parts with the zero boundary conditions provided by the bubble function~$\bE$, and write
\[
\Vert \curlbold \, \curlbold \rhobold \Vert^2_{0,\E} \approx \int_\E \bE (\curlbold \, \curlbold \rhobold)^2 =  \int_\E  \curlbold (\bE \curlbold \curlbold \rhobold) \cdot \curlbold \rhobold.
\]
Next, a Cauchy-Schwarz inequality, the polynomial inverse inequality~\eqref{H1-L2}, and a trivial bound on the~$L^{\infty}$ norm of the bubble function~$\bE$ yield
\[
\begin{split}
\Vert \curlbold \, \curlbold \rhobold \Vert^2_{0,\E} 
& \lesssim \Vert \curlbold (\bE \curlbold\, \curlbold \rhobold) \Vert_{0,\E}  \Vert  \curlbold \rhobold \Vert_{0,\E} \\
& \lesssim \hE^{-1} \Vert \curlbold\, \curlbold \rhobold \Vert_{0,\E} \Vert  \curlbold \rhobold \Vert_{0,\E},
\end{split}
\]
which entails~\eqref{3D:inverse-curl}.

Inserting~\eqref{3D:inverse-curl} in~\eqref{bound-on-rhobold} and using bound on~$\Psi$ in~\eqref{bound:Psi-face} give
\begin{equation} \label{bound:rhobold-face}
\Vert \curlbold \rhobold \Vert_{0,\E}  \lesssim  \vert \Psi \vert_{1,\E} \lesssim \hE^{\frac{1}{2}}   \Vert \psiboldh \cdot \nE \Vert_{0,\partial \E}.
\end{equation}
Combining~\eqref{bound:Psi-face} and~\eqref{bound:rhobold-face} into~\eqref{using:orthogonality:face} yields the assertion.
\end{proof}

Thanks to the a priori estimate~\eqref{a-priori:bound:psiboldh}, we can prove the following interpolation result.
\begin{prop} \label{prop:best-interpolation-face}
For all~$\E \in \taun$, let~$\psibold \in [H^s(\E)]^3$, $ 1/2 < s \le 1$. Then, there exists~$\psiboldI \in \VzfaceE$ such that
\[
\Vert \psibold - \psiboldI \Vert_{0,\E} \lesssim  \hE^s \vert  \psibold \vert_{s,\E}.
\]
\end{prop}
\begin{proof}
Given~$\E \in \taun$, let~$\psiboldpi$ be the vector given by the average of each component of~$\psibold$ over~$\E$.
We define the interpolant~$\psiboldI \in \VzfaceE$ through its degrees of freedom as follows:
\begin{equation} \label{definition:face-interpolant}
\int_\F (\psibold - \psiboldI ) \cdot \nF = 0 \quad \quad \forall \F \in \EE.
\end{equation}
Note that~\eqref{definition:face-interpolant} is well defined since~$\psibold$ belongs to~$[H^s(\E)]^3$ with~$s>1/2$.

The triangle inequality, the definition of~$\psiboldpi$, and the Poincar\'e inequality together with interpolation theory yield
\[
\Vert \psibold - \psiboldI \Vert_{0,\E} \le \Vert \psibold - \psiboldpi \Vert_{0,\E} +  \Vert \psiboldI - \psiboldpi \Vert_{0,\E} \lesssim \hE^s \vert  \psibold \vert_{s,\E} +  \Vert \psiboldI - \psiboldpi \Vert_{0,\E}.
\]
We still have to bound the second term on the right-hand side.
Since vector constant functions are contained in the space~$\VzfaceE$, the term~$\psiboldI-\psiboldpi$ belongs to~$\VzfaceE$.
To get the assertion, we apply~\eqref{a-priori:bound:psiboldh}, the triangle inequality, the definition of~$\psiboldI$ in~\eqref{definition:face-interpolant}, and inequality~\eqref{Agmon:zero-average}:
\[
\begin{split}
\Vert \psiboldI - \psiboldpi \Vert_{0,\E} 
& \lesssim \hE^{\frac{1}{2}} \Vert  (\psiboldI - \psiboldpi)\cdot \nE \Vert_{0,\partial \E} \\
& \le  \hE^{\frac{1}{2}} \Vert  (\psibold - \psiboldpi)\cdot \nE \Vert_{0,\partial \E} + \hE^{\frac{1}{2}} \Vert  (\psibold - \psiboldI )\cdot \nE \Vert_{0,\partial \E} \\
& \le 2 \hE^{\frac{1}{2}} \Vert  (\psibold - \psiboldpi)\cdot \nE \Vert_{0,\partial \E} \lesssim \hE^s \vert \psibold \vert_{s,\E} ,
\end{split}
\]
whence the assertion follows.
\end{proof}

For future convenience, we also define face virtual element spaces on faces.
More precisely, let~$\F \in \Ecaln$ be a face. We define the low order local 2D face virtual element space as follows:
\begin{equation} \label{2d-face-space}
\begin{split}
\VzfaceF := \{ \psiboldh \in [L^2(\F)]^2 \mid 	& \divF \psiboldh \in \Pbb_0(\E),\, \rotF \psiboldh \in \Pbb_0(\E),    \\ 
								& \psiboldh{}_{|\e} \cdot \ne \in \Pbb_0(\e) \, \forall \e \in \EF , \, \int_\E \psiboldh \cdot \xboldF^\perp=0  \},
\end{split}
\end{equation}
where we have set, in the local coordinates of~$\F$,
\begin{equation} \label{def:xboldF}
\xboldF := (x-\xF, y - \yF)^T \ , \qquad \xboldF^\perp := (\yF - y, x-\xF)^T .
\end{equation}

\begin{cor}
For all~$\E \in \taun$, let~$\psibold \in [H^\varepsilon(\E)]^3$, $\varepsilon > 0$, and~$\div \psibold \in H^s(\E)$, $ 0 \le s \le 1$.
Then, the following bound is valid:
$$
\Vert \div (\psibold - \psiboldI) \Vert_{0,\E} \lesssim  \hE^s \vert  \div \psibold \vert_{s,\E}.
$$
Such result, combined with Proposition \ref{prop:best-interpolation-face}, yields an optimal $O(\hE^s)$ interpolation estimate in the $H_{\div}$ norm.
\end{cor}
\begin{proof}
It is trivial to check that $\div \psiboldI \in \Pbb_0(K)$ corresponds to the average of~$\div \psibold$ on~$\E$.
The assertion follows from the Poincar\'e inequality. The regularity condition above on $\psibold$ is stated also to guarantee the well-posedness of the interpolant.  
\end{proof}

\section{Interpolation properties of 3D edge spaces}  \label{section:edge3DVEM-loworder}

\subsection{A useful result for 2D edge spaces}  \label{section:edge2DVEM-loworder}
Let~$\E \in \taun$ be an element and~$\F \in \EE$ be any of its faces. We define the low order local 2D edge virtual element space as follows:
\begin{equation} \label{2d-edge-space}
\begin{split}
\VzedgeF := \{ \vboldh \in [L^2(\F)]^2 \mid 	& \divF \vboldh \in \Pbb_0(\F),\, \rotF \vboldh \in \Pbb_0(\F),    \\ 
									& \vboldh{}_{|\e} \cdot \te \in \Pbb_0(\e) \, \forall \e \in \EF , \, \int_\F \vboldh \cdot \xboldF=0  \},
\end{split}
\end{equation}
where~$\xboldF $ is defined in~\eqref{def:xboldF}.

In what follows, \emph{enhancing constraint} refers to the constraint
\begin{equation} \label{enhancing:edge2D}
 \int_\F \vboldh \cdot \xboldF=0 .
\end{equation}
We endow the local space~$\VzedgeF$ with a set of unisolvent degrees of freedom, which is provided by the single value of~$\vboldh \cdot \te$ over each edge~$\e \in \EF$.
Moreover, it contains 2D N\'ed\'elec polynomials of degree~$1$, i.e, the space~$[\Pbb_0(\E)]^2 \oplus (\xbold^{\perp} \Pbb_0(\E))$.

The main result of the section is the following a priori estimate bound.
Although its proof has some similarities with that of Proposition~\ref{prop:splitting:VEM-face}, it is worth reporting the details here for completeness.
\begin{prop} \label{prop:splitting:VEM-2Dedge}
Given~$\vboldh \in \VzedgeF$, the following a priori bound is valid:
\begin{equation} \label{a-priori:bound:vboldh2D}
\Vert \vboldh  \Vert_{0,\F} \lesssim \hF^{\frac{1}{2}}  \Vert \vboldh \cdot \tpartialF \Vert_{0,\partial \F}.
\end{equation}
\end{prop}
\begin{proof}
We first prove that there exist~$\rho \in H^1(\F) \slash \Rbb$ and~$\Psi \in H^1(\F)$ such that the following Helmholtz decomposition is valid:
\begin{equation} \label{splitting:VEM-2Dedge}
\vboldh = \curlboldF \rho + \nablaF \Psi.
\end{equation}
Define~$\rho \in H^1(\F) \slash \Rbb$ and~$\Psi \in H^1(\F)$ as the weak solutions to the problems
\begin{equation} \label{definition:rho:2Dedge}
\begin{cases}
\DeltaF \rho = \rotF \vboldh 									& \text{in } \F \\
 \tpartialF \cdot \curlboldF \rho (= \npartialF \cdot \nablaF \rho) = \tpartialF \cdot \vboldh 	& \text{on } \partial \F
\end{cases}
\end{equation}
and
\[
\begin{cases}
\DeltaF \Psi = \divF \vboldh 	& \text{in } \F \\
\Psi = 0 					& \text{on } \partial \F.
\end{cases}
\]
It can be readily checked that~$\rho$ and~$\Psi$ are defined so that~\eqref{splitting:VEM-2Dedge} is valid. 
Moreover, we also have the orthogonality condition $(\curlboldF \rho, \nablaF \Psi)_{0,\F} = 0$, whence we get
\begin{equation} \label{using:orthogonality:2Dedge}
\Vert \vboldh  \Vert^2_{0,\F} = \Vert \curlboldF \rho \Vert^2_{0,\E} + \Vert \nablaF \Psi \Vert ^2_{0,\E}.
\end{equation}
We bound the two terms on the right-hand side of~\eqref{using:orthogonality:2Dedge} separately. We begin with that involving the~$\curlboldF$ operator.
To this aim, we preliminarily recall the identity
\begin{equation} \label{rot-curl:Laplace}
\rotF \, \curlboldF = \DeltaF.
\end{equation}
Integrating by parts and using differential identity~\eqref{rot-curl:Laplace} together with the fact that~$\rho$ has constant Laplacian and zero average over~$\F$, we obtain
\[
\Vert \curlboldF \rho \Vert_{0,\F}^2 = -\int_\F \rho \,  \DeltaF \rho + \int_{\partial \F} \rho (\tpartialF \cdot \curlboldF \rho )  = \int_{\partial \F} \rho (\tpartialF \cdot \curlboldF \rho ),
\]
which, recalling~\eqref{definition:rho:2Dedge}, applying a Cauchy-Schwarz inequality, and using~\eqref{Agmon:zero-average} with~$s=1$, yields
\[
\Vert \curlboldF \rho \Vert_{0,\F}^2 \le \Vert \rho \Vert_{0,\partial \F} \Vert \tpartialF \cdot \vboldh \Vert_{0,\partial \F} \lesssim \hE^{\frac{1}{2}} \Vert \nablaF \rho \Vert_{0, \F} \Vert \tpartialF \cdot \vboldh  \Vert_{0,\partial \F}.
\]
Observing that~$\Vert \nablaF \rho \Vert_{0,\F} = \Vert \curlboldF \rho \Vert_{0,F}$, we get
\begin{equation} \label{bound:curlrho-2Dedge}
\Vert \curlboldF \rho \Vert_{0,\F}  \lesssim \hF^{\frac{1}{2}} \Vert \tpartialF \cdot \vboldh \Vert_{0,\partial \F}.
\end{equation}
Next, we deal with the second term on the right-hand side of~\eqref{using:orthogonality:2Dedge}.
To this aim, we use an integration by parts, $\Psi_{|\partial \F}= 0$, the fact that~$\DeltaF \Psi$ is constant over~$\F$, and the identity~$2\divF  \xboldF =1$:
\[
\Vert \nablaF \Psi \Vert^2_{0,\F} = \int_\F (-\DeltaF \Psi) \Psi + \int_{\partial \F} \Psi ( \nablaF \Psi \cdot \npartialF) =\frac{1}{2} \int_\F (-\DeltaF \Psi) \Psi \divF \xboldF.
\]
First an integration by parts, then equation~\eqref{splitting:VEM-2Dedge} and the enhancing constraint~\eqref{enhancing:edge2D}, finally a Cauchy-Schwarz inequality lead to
\begin{equation} \label{grad:Psi}
\begin{split}
\Vert \nablaF \Psi \Vert^2_{0,\F} 	& = \frac{1}{2} \int_\F (\DeltaF \Psi) \nablaF \Psi \cdot \xboldF = \frac{1}{2} \int_\F (\DeltaF \Psi) \left(\vboldh - \curlboldF \rho \right) \cdot \xboldF  \\
						& =  -\frac{1}{2} \int_\F (\DeltaF \Psi)  \curlboldF \rho  \cdot \xboldF  \lesssim \Vert \DeltaF \Psi \Vert_{0,\F} \Vert \curlboldF \rho \Vert_{0,\F} \Vert \xboldF \Vert_{\infty, \F}.
\end{split}
\end{equation}
The following inverse estimate is valid:
\begin{equation} \label{2D:inverse-Laplacian}
\Vert \DeltaF \Psi \Vert_{0,\F} \lesssim \hF^{-1} \Vert \nablaF \Psi \Vert_{0,\F}.
\end{equation}
To prove~\eqref{2D:inverse-Laplacian}, introduce~$\bF$ the piecewise cubic bubble function over the decomposition of~$\F$ into triangles~$T$; see the proof of Lemma~\ref{lemma:polynomial-inverse}.
Using that~$\DeltaF \Psi$ is constant over~$\F$ and the equivalence of norms in finite dimensional spaces on each triangle T yields
\[
 \Vert \DeltaF \Psi \Vert_{0,\F}^2 \approx \int_\F \bF (\DeltaF \Psi)^2.
\]
Integrating by parts, using that~$\bF{}_{|\partial \F}=0$, applying the polynomial inverse inequality~\eqref{H1-L2}, recalling that~$\Vert  \bE \Vert_{\infty ,\E} =1$, and using again the fact that~$\DeltaF \Psi$ is constant over~$\F$ allow us to write
\[
\begin{split}
 \Vert \DeltaF \Psi \Vert_{0,\F}^2
& \approx  \int_\F \divF (\bF  \DeltaF \Psi) \cdot \nablaF \Psi \le \Vert \divF (\bF \DeltaF \Psi)  \Vert_{0,\F} \vert \Psi  \vert_{1,\F} \\
& \lesssim \hF^{-1} \Vert \bF \DeltaF \Psi  \Vert_{0,\F} \vert \Psi  \vert_{1,\F} \le \hF^{-1} \Vert  \DeltaF \Psi  \Vert_{0,\F} \vert \Psi  \vert_{1,\F} ,
\end{split}
\]
whence~\eqref{2D:inverse-Laplacian} follows.

Inserting the inverse inequality~\eqref{2D:inverse-Laplacian}, the bound~$\Vert \xboldF \Vert_{\infty, \F} \lesssim \hF$, and~\eqref{bound:curlrho-2Dedge} in~\eqref{grad:Psi} yields
\begin{equation} \label{bound:Psi-2Dedge}
\Vert \nabla \Psi \Vert_{0,\F} \lesssim \Vert \curlboldF \rho \Vert_{0,\F} \lesssim 
\hF^{\frac{1}{2}} \Vert \tpartialF \cdot \vboldh \Vert_{0,\partial \F} .
\end{equation}
Inserting~\eqref{bound:curlrho-2Dedge} and~\eqref{bound:Psi-2Dedge} in~\eqref{using:orthogonality:2Dedge} provides~\eqref{a-priori:bound:vboldh2D}, which concludes the proof.
\end{proof}

\subsection{The main interpolation result for edge spaces}
Henceforth, given a vector function~$\vbold: \E \rightarrow \Rbb^3$ and any face~$\F\in \EE$, the symbol~$\vboldF$ denotes the vector field given by the tangent component of~$\vboldF$ on~$\F$,
i.e., the restriction to~$\F$ of $\vbold - \nF (\vbold \cdot \nF)$.
We can interpret such a function as a three-component vector field, but also as a two-component vector field living in the tangent plane to the face.
With an abuse of notation, we shall use both interpretations in the following. Finally, observe that $\vboldF$ corresponds to a~$90^{\circ}$ rotation of~$\vboldF \times \nF$.
We also introduce a function~$\vboldpartialE$ over~$\partial\E$ defined as~$\vboldpartialE{}_{|\F} :=\vboldF$ on each face~$\F$.

Given~$\E \in \taun$, following~\cite{da2018lowest}, we define the low order local 3D edge virtual element space in~$\E$ as follows:
\[
\begin{split}
\VzedgeE := \{ \vboldh 	& \in [L^2(\E)]^3 \mid  \div \vboldh =0 ,\, \curlbold \,\curlbold \vboldh \in [\Pbb_0(\E)]^3 , \\ 
					& \vboldhF  \in \VzedgeF \, \forall \F \in \EE , \, \vboldh \cdot \widetilde{\te} \text{ continuous  at each edge~$\e$ of~$\E$}, \\
					& \int_\E \curlbold \vboldh \cdot (\xboldE \times \qbold_0) =0 \, \forall \qbold_0 \in [\Pbb_0(\E)]^3    \},
\end{split}
\]
where~$\xboldE$ is defined as in~\eqref{def:xboldE}, and $\widetilde{\te}$ denotes a single valued unit tangent vector to edge $e$ chosen once and for all.

We endow the local space~$\VzedgeE$ with a set of unisolvent degrees of freedom, which is provided by the single value of~$\vboldh \cdot \te$ over each edge~$\e$ of~$\E$.
Moreover, it contains 3D N\'ed\'elec polynomials of degree~$1$, i.e, the space~$[\Pbb_0(\E)]^3 \oplus (\xbold \times \Pbb_0(\E))$.

In what follows, we use the following exact sequence property, which was proven in~\cite[Proposition~4.2]{da2018lowest}:
\begin{equation} \label{curlbold-edge=face}
\curlbold (\VzedgeE) = \{ \psiboldh \in \VzfaceE \mid \div \psiboldh = 0    \}  \subset \VzfaceE  .
\end{equation}
With this result at hand, we are able to prove the following bound for functions in 3D edge virtual element spaces.
\begin{prop} \label{proposition:apriori-3Dedge}
The following bound is valid for any~$\vboldh \in \VzedgeE$:
\begin{equation} \label{apriori-3Dedge}
\Vert  \vboldh  \Vert_{0,\E} \lesssim \hE^{\frac{1}{2}} \Vert  \vboldhF   \Vert_{0,\partial \E}.
\end{equation}
\end{prop}
\begin{proof}
{Recall the differential identity~$\curlbold \vboldh\cdot \nF{}_{|\F} = \rotF \vboldhF$ on each face~$\F \in \EE$.}

Using~\eqref{curlbold-edge=face} and~\eqref{a-priori:bound:psiboldh} for 3D face virtual element spaces, we readily get
\[
\Vert \curlbold \vboldh \Vert_{0,\E} \lesssim \hE^{\frac{1}{2}} \sum_{\F \in \EE} \Vert  \curlbold \vboldh \cdot \nF \Vert_{0,\F} = \hE^{\frac{1}{2}} \sum_{\F \in \EE} \Vert  \rotF \vboldhF \Vert_{0, \F} .
\]
The following inverse estimate is valid:
\begin{equation} \label{inverse-rot}
\Vert  \rotF \vboldhF  \Vert_{0, \F} \lesssim \hE^{-1}  \Vert  \vboldhF \Vert_{0, \F}.
\end{equation}
Bound~\eqref{inverse-rot} can be shown using the same arguments employed to prove~\eqref{3D:inverse-curl} and~\eqref{2D:inverse-Laplacian}, the essential ingredient being the fact that~$\vboldhF$ is a 2D edge virtual element function,
which in turns implies that~$\rotF \vboldhF$ is constant over~$\F$.

We deduce
\[
\Vert \curlbold \vboldh \Vert_{0,\E} \lesssim  \hE^{-\frac{1}{2}} \sum_{\F \in \EE} \Vert  \vboldhF \Vert_{0, \F}.
\]
Assertion~\eqref{apriori-3Dedge} follows by inserting the above inequality on the right-hand side of the Friedrichs' inequality~\eqref{Friedrichs-Monk} and summing over the faces of~$\E$.
\end{proof}

We collect estimates~\eqref{a-priori:bound:vboldh2D} and~\eqref{apriori-3Dedge} in the following result, stating that the opportunely scaled degrees of freedom uniformly control the~$L^2$ norm of functions in~$\VzedgeE$.
\begin{cor}
The following a priori bound is valid for any~$\vboldh \in \VzedgeE$:
\begin{equation} \label{apriori-3Dedge-full}
\Vert  \vboldh  \Vert_{0,\E}^2 \lesssim \hE^2 \sum_{\F \in \EE} \sum_{\e\in \EF}  \Vert \vboldh \cdot \te \Vert_{0,\e}^2.
\end{equation}
\end{cor}

Next, we show interpolation properties for 3D edge virtual element spaces. To this aim, we need a preliminary technical result.
\begin{lem} \label{lemma:L1-boundary-estimate}
Let~$\F \in \Fcaln$. Then, for all~$\vbold \in [H^{\varepsilon}(\F)]^2 \cap H(\rotF, \F)$, with~$\varepsilon \in (0, 1/2]$, such that its tangent component on each edge is integrable, the following bound is valid:
\begin{equation} \label{tricky-bound-L1}
\sum_{\e \in\EF} \left\vert \int_\e \vbold \cdot \tpartialF \right\vert^2
\lesssim \Vert \vbold \Vert_{0,\F}^2 + \hF^{2\varepsilon} \vert \vbold \vert_{\varepsilon, \F}^2 + \hF^2 \Vert \rotF \vbold \Vert_{0,\F}^2.
\end{equation}
\end{lem}
\begin{proof}
The first and main step in the proof is showing the following trace-type inequality: for all fixed~$p>2$,
\begin{equation}\label{eq:int-to-inside}
\sum_{\e \in\EF} \left\vert \int_\e \vbold \cdot \tpartialF \right\vert
\lesssim \hF^{1-\frac{2}{p}} \Vert \vbold \Vert_{L^p(\F)} + \hF \Vert \rotF \vbold \Vert_{0,\F}.
\end{equation}
We split the face~$\F$ into a shape-regular triangulation~$\tautilden(\F)$, e.g., as in the proof of Lemma~\ref{lemma:polynomial-inverse}.
It suffices to prove, for all~$\e \in \EF$,
\begin{equation} \label{it-suffices-to-prove}
\left\vert \int_\e \vbold \cdot \te \right\vert
\lesssim \hF^{1-\frac2p} \Vert \vbold \Vert_{L^p(T)} + \hF \Vert \rotF \vbold \Vert_{0,T},
\end{equation}
where~$T$ is the only triangle in~$\tautilden(\F)$ such that~$\e \subset \partial T$ (and we recall $h_T \sim\hF$).

Let~$\That$ be the reference triangle, $T$ be mapped to~$\That$ through the (rotated) Piola transform, and~$\ehat$ be the edge of~$\That$ corresponding to the edge~$\e$ through the (rotated) Piola transform.
The trace theorem on Lipschitz domains states that the trace operator is surjective from~$W^{1,p'}(\That)$ to~$W^{\frac1p, p'}(\partial \That)$.
Further, the space~$W^{\frac1p,p'}(\partial \That)$ contains piecewise discontinuous functions over~$\partial \That$ for~$p>2$.
In particular, there exists a function~$\what$ such that~$\what =1$ on~$\ehat$, $\what=0$ on~$\partial \That \setminus \ehat$,
and~$\Vert \what \Vert_{W^{1,p'}(\That)}<\infty$ (we pick one among the many).

Next, we use a scaling argument (the $\widehat \cdot$ denoting the usual pull-back),
an integration by parts, the H\"older inequality, and simple manipulations, and get
\[
\begin{split}
\left\vert \int_\e \vbold\cdot \te \right\vert
& \lesssim \hF \left\vert \int_{\ehat} \vboldhat \cdot \thatehat \right\vert
  = \hF  \left\vert \int_{\partial\That} (\vboldhat \cdot \thatehat) \what \right\vert 
  = \hF \left\vert \int_{\That} \rothatFhat \vboldhat \ \what
        + \int_{\That} \vboldhat \cdot \widehat{\curlbold}_{\hat F} \what \right\vert \\
& \lesssim \hF \left( \Vert \rothatFhat \vboldhat \Vert_{0,\That} \Vert \what \Vert_{0,\That} 
  + \Vert \vboldhat \Vert_{L^p(\That)} \vert \what \vert_{W^{1,p'}(\That)} \right)\\
& \lesssim \hF \left( \Vert \vboldhat \Vert_{L^p(\That)} + \Vert \rothatFhat \vboldhat \Vert_{0,\That} \right) \Vert \what \Vert_{W^{1,p'}(\That)}
  \lesssim \hF^{1-\frac2p} \Vert \vbold \Vert_{L^p(T)} + \hF \Vert \rotF \vbold \Vert_{0,T}. 
\end{split}
\]
We have shown \eqref{eq:int-to-inside}. The assertion finally follows by picking~$p = 2/(1-\varepsilon) > 2$ in \eqref{eq:int-to-inside}, using the embedding~$H^\varepsilon(\F) \hookrightarrow  L^p(\F)$, and performing standard manipulations.
\end{proof}

Introduce the following scaled Sobolev norms on faces: given~$\varepsilon>0$,
\begin{equation} \label{scaled-norm-faces}
\Vertiii{\cdot}_{\varepsilon, \F}^2 := \Vert \cdot \Vert^2_{0,\F} + \hF^{2\varepsilon} \vert \cdot \vert_{\varepsilon,\F}^2 \quad \quad \forall \F \in \Fcaln.
\end{equation}
The norm in~\eqref{scaled-norm-faces} is defined so that it scales like an $L^2$ norm on faces.

We are now in the position of showing the following interpolation result.

\begin{prop} \label{prop:best-interpolation-edge}
For all~$\E \in \taun$, let~$\vbold \in H^s(\curlbold,\E)$, $1/2 < s \le 1$, such that its tangent component on each edge is integrable.
Then, there exists~$\vboldI \in \VzedgeE$ such that
\[
\Vert \vbold - \vboldI \Vert _{0,\E} \lesssim \hE^{s} \vert \vbold \vert_{s, \E} + \hE \Vert \curlbold \vbold \Vert _{0,\E} + \hE^{s + 1} \vert \curlbold \vbold \vert_{s, \E} \quad \quad \forall \E \in \taun.
\]
\end{prop}
\begin{proof}
Define~$\vboldI$ as the interpolant in~$\VzedgeE$ of~$\vbold$, i.e.,
\begin{equation} \label{edge-interpolant}
\int_\e (\vbold - \vboldI) \cdot \widetilde{\te} = 0 \quad \quad \forall \text{ edges~$\e$ of~$\E$}.
\end{equation}
It is easy to check that the integral in~\eqref{edge-interpolant} is well defined for~$\vbold \in H^s(\curlbold,\E)$, $s > 1/2$.
Let~$\vboldpi \in [\Pbb_0(\E)]^3$ be the vector average of~$\vbold$ over~$\E$, which is contained in $\VzedgeE$. 
Applying the triangle inequality and~\eqref{apriori-3Dedge-full}, we get
\[
\Vert \vbold - \vboldI \Vert^2_{0,\E} \lesssim \Vert \vbold - \vboldpi \Vert^2_{0,\E} + \hE^2 \sum_{\F \in \EE} \sum_{\e\in \EF} \Vert (\vboldI - \vboldpi) \cdot \te   \Vert_{0,\e}^2.
\]
Using that~$\vboldI \cdot \te$ is constant on each edge~$\e$ and~\eqref{edge-interpolant}, we write
\[
\Vert (\vboldI - \vboldpi) \cdot \te   \Vert_{0,\e}^2
= \frac{1}{\he} \left| \int_\e (\vboldI - \vboldpi) \cdot \te \right|^2
= \frac{1}{\he} \left| \int_\e (\vbold - \vboldpi) \cdot \te \right|^2.
\]
Combining the two above equations and using $\he \sim \hE$ entail
\begin{equation} \label{splitting:T1-T2-T3}
\Vert \vbold - \vboldI \Vert^2_{0,\E} 
\lesssim \Vert \vbold - \vboldpi \Vert^2_{0,\E} + \hE \sum_{\F \in \EE} \sum_{\e\in \EF} \left| \int_\e (\vbold - \vboldpi) \cdot \te \right|^2
=: T_1 + T_2.
\end{equation}
We start by showing an upper bound for the term $T_2$. Given~$\varepsilon>0$, applying~\eqref{tricky-bound-L1}, a (scaled) trace inequality, and the Poincar\'e inequality together with interpolation theory, we get
\begin{equation} \label{initial:bound:T2}
\begin{split}
T_2 
& \lesssim \hE \sum_{\F \in \EE} \left( \Vertiii{\vbold - \vboldpi}^2_{\varepsilon,\F} + \hE^2 \Vert \rotF \vboldhF \Vert_{0,\F}^2  \right) \\
& \lesssim \hE^{1+2 \varepsilon} \vert \vbold \vert^2_{1/2 + \varepsilon, \E} + \hE^3 \sum_{\F \in \EE} \Vert \rotF \vboldhF \Vert^2_{0,\F} ,
\end{split}
\end{equation}
where the scaled norm~$\Vertiii{\cdot}_{\varepsilon, \F}$ is defined in~\eqref{scaled-norm-faces}.

We bound the second term on the right-hand side of~\eqref{initial:bound:T2}. Upon rewriting the~$\rotF$ operator on each face as the normal component of the~$\curlbold$ operator, we apply~\eqref{Agmon} and deduce
\begin{equation} \label{bound:rot}
\hE^3 \sum_{\F \in \EE} \Vert \rotF \vboldhF \Vert^2_{0,\F} \lesssim \hE^2 \Vert \curlbold \vbold \Vert^2_{0,\E} + \hE^{2\varepsilon + 3} \vert \curlbold \vbold \vert_{\varepsilon+1/2,\E}^2 .
\end{equation}
Inserting~\eqref{bound:rot} into~\eqref{initial:bound:T2} and choosing~$s=\varepsilon + 1/2$ yield
\begin{equation} \label{bound:T2}
T_2 \lesssim \hE^{2 s} \vert \vbold \vert^2_{s, \E} +   \hE^{2} \Vert \curlbold \vbold \Vert^2_{0,\E}  + \hE^{2+2s} \vert \curlbold \vbold \vert_{s,\E}^2 .
\end{equation}
The assertion follows by collecting~\eqref{bound:T2} in~\eqref{splitting:T1-T2-T3} and applying the Poincar\'e inequality on the term~$T_1$.

\end{proof}

Although the result in Proposition \ref{prop:best-interpolation-edge} requires some ``additional regularity'' for the objective function, this is essentially needed in order to make the interpolation operator well defined.
Combining such bound with an approximation estimate for the $\curlbold$, yield estimates in the $H_{\curlbold}$ norm that are optimal also with respect to the required regularity, as shown in the corollary below.

\begin{cor}\label{corcorcor} 
For all~$\E \in \taun$, let~$\vbold \in H^s(\curlbold,\E)$, $ 1/2 < s \le 1$ such that  its tangent component on each edge is integrable.
Then, the following bound is valid:
$$
\Vert \curlbold (\vbold - \vboldI) \Vert_{0,\E} \lesssim  \hE^s \vert  \curlbold \vbold \vert_{s,\E}.
$$
Such result, combined with Proposition \ref{prop:best-interpolation-edge}, yields an optimal $O(\hE^s)$ interpolation estimate in the $H_{\curlbold}$ norm.
\end{cor}
\begin{proof}
Thanks to the exact sequence property~\eqref{curlbold-edge=face} and the definition of the interpolation operators, it is easy to check that the following commuting diagram property is valid.
For all $\vbold \in H^s(\curlbold,\E)$, $ 1/2 < s \le 1$,
\[
\curlbold \vboldI = ( \curlbold \vbold )_I \, ,   
\]
where the interpolation on the left-hand side is into~$\VzedgeE$ and the interpolation on the right-hand side is into~$\VzfaceE$.
The result follows from the above identity and Proposition~\ref{prop:best-interpolation-face} with~$\psibold = \curlbold \vbold$.
\end{proof}

\begin{remark} \label{remark:curl-interpolation}
The condition that the tangent component of ${\bf v}$ is integrable on each edge, appearing in Lemma \ref{lemma:L1-boundary-estimate}, Proposition \ref{prop:best-interpolation-edge}, and Corollary \ref{corcorcor},
is the minimal one in order to define the interpolant ${\bf v}_I$ in~\eqref{edge-interpolant}
by simple edge integrals.
We could avoid such a condition by interpreting~$\int_\e {\bf v}\cdot \te$ in the dual sense rather than as an integral,
i.e., $\langle {\bf v}\cdot \te , \chi_e\rangle_{W^{1/p,p'}(\partial \F)}$ with $\chi_e: \partial \F \rightarrow {\mathbb R}$ equal to~$1$ on~$e$ and vanishing on~$\partial \F \setminus \e$.
\end{remark}

\section{Stability properties of discrete scalar products}  \label{section:stabilizations}

Face and edge virtual element functions are typically associated with the computation of $L^2$ inner products over elements.
For instance, this is the case for the discretization of magnetostatic problems~\cite{da2018lowest, da2018family} and the Maxwell's equations~\cite{MaxwellVEM}.
Since functions in~$\VzfaceE$ and~$\VzedgeE$ are not available in closed-form, we address the computation of $L^2$ bilinear forms following the VEM gospel~\cite{VEMvolley}.
The standard VEM construction of $L^2$ scalar products is as follows: given any two (face or edge) virtual element functions~$\uboldh$ and~$\vboldh$,
\[
(\uboldh, \vboldh) _{0,\E} \approx  ( \Piboldz  \uboldh, \Piboldz \vboldh) _{0,\E} + \SE((\Ibold - \Piboldz) \uboldh, (\Ibold - \Piboldz) \vboldh),
\]
where~$\Piboldz$ is an $L^2$-orthogonal projector into a polynomial space, $\Ibold$ is the vector identity operator,
and~$\SE$ is a bilinear form scaling like the $L^2$ norm on the virtual element functions.
More precisely, it is required the existence of two positive constants $\alpha_* \le \alpha^*$ such that, for all functions~$\vboldh$ in the virtual element space under consideration,
\begin{equation} \label{bound:+}
\alpha_* \Vert \vboldh \Vert_{0,\E}^2 \le \SE( \vboldh, \vboldh) \le \alpha^* \Vert \vboldh \Vert_{0,\E} ^2.
\end{equation}
We require the projections~$\Piboldz$ and the stabilizations~$\SE(\cdot, \cdot)$ to be computable via the degrees of freedom of the virtual element space under consideration.

In the literature of the VEM, bounds of the form~\eqref{bound:+} have been analyzed for nodal elements~\cite{brennerVEMsmall, BrennerGuanSung_someestimatesVEM, beiraolovadinarusso_stabilityVEM, chen_anisotropic_conforming, chen_anisotropic_nonconforming},
but no result exists for face and edge elements.
Here, we derive an \emph{explicit} analysis of the stability bounds~\eqref{bound:+} for VE face and edge spaces.
We cope with face and edge virtual element spaces in Sections~\ref{subsection:stab-face} and~\ref{subsection:stab-edge} below, respectively.

For future convenience, we state here a preliminary, technical result.
\begin{lem} \label{lemma:inverse1D}
Given a face~$\F$, let~$\vboldh \in \VzedgeF$. Then, the following inverse estimate is valid:
\begin{equation} \label{inverse1D}
\Vert \vboldh \cdot \tpartialF \Vert_{0,\partial \F} \lesssim \hF^{-\frac{1}{2}} \Vert \vboldh \cdot \tpartialF \Vert_{-\frac{1}{2}, \partial \F},
\end{equation}
where as usual~$\tpartialF{}_{|\e}:=\te$ and~$\Vert \cdot \Vert_{-\frac12,\partial\F}$ is defined in~\eqref{H-1/2:norm}.
\end{lem}
\begin{proof}
It suffices to recall that~$\vboldh \cdot \tpartialF $ is piecewise constant over each edge~$e \in \EF$ and apply essentially the same techniques as in Lemma~\ref{lemma:polynomial-inverse}, in 2D instead of in 3D.
\end{proof}

\subsection{Stability bounds for face virtual elements} \label{subsection:stab-face}
As for face virtual elements, consider the $L^2$ projector~$\Piboldface: \VzfaceE \rightarrow [\Pbb_0(\E)]^3$, which is computable via the degrees of freedom, i.e., the value of the normal component over faces.
Consider the stabilization, cf. \cite[equation~$(4.17)$]{da2018lowest},
\begin{equation} \label{stabilization:face}
\SEface(\psiboldh, \phiboldh) := \hE \sum_{\F \in \EE} (\psiboldh \cdot \nF, \phiboldh \cdot \nF)_{0,\F} \quad \quad \forall \psiboldh, \phiboldh \in \VzfaceE.
\end{equation}
The stabilization~$\SEface(\cdot, \cdot)$ is explicitly computable via the degrees of freedom since $\psiboldh |_F \cdot \nF{}$ and $\phiboldh |_F \cdot \nF{}$ belong to~$\Pbb_0(\F)$ for all faces~$\F$.

We prove the following stability result.
\begin{prop} \label{proposition:stab-face}
Given~$\SEface$ the stabilization in~\eqref{stabilization:face}, the following bounds are valid:
\begin{equation} \label{stability-bounds-faces}
\alpha_*\Vert  \psiboldh  \Vert^2_{0,\E} \le   \SEface(\psiboldh, \psiboldh)   \le   \alpha^* \Vert  \psiboldh  \Vert^2_{0,\E}     \quad \quad \forall \psiboldh \in \VzfaceE,
\end{equation}
where~$\alpha_*$ and~$\alpha^*$ are two positive constants independent of~$\hE$.
\end{prop}
\begin{proof}
The lower bound in~\eqref{stability-bounds-faces} follows immediately from~\eqref{a-priori:bound:psiboldh}.

As for the upper bound in~\eqref{stability-bounds-faces}, we apply the polynomial inverse inequality~\eqref{L2-H1/2}, the $\div$-trace inequality~\eqref{div-trace}, and get
\begin{equation} \label{one-step:stability-face}
\Vert \psiboldh \cdot \nE \Vert_{0,\partial \E} \lesssim \hE^{-\frac{1}{2}} \Vert \psiboldh \cdot \nE \Vert_{-\frac{1}{2}, \partial \E} \lesssim \hE^{-\frac{1}{2}} \Vert \psiboldh \Vert_{0,\E} + \hE^{\frac{1}{2}} \Vert \div \psiboldh \Vert_{0,\E}.
\end{equation}
The following inverse inequality is valid:
\begin{equation} \label{inverse-ineq-div}
\Vert \div \psiboldh \Vert_{0,\E}  \lesssim \hE^{-1} \Vert \psiboldh \Vert_{0,\E}.
\end{equation}
To prove~\eqref{inverse-ineq-div}, it suffices to recall that~$\div \psiboldh \in \Pbb_0(\E)$ and proceed as in the proof of inverse estimates~\eqref{3D:inverse-curl}, \eqref{2D:inverse-Laplacian}, and~\eqref{inverse-rot}.
Inserting~\eqref{inverse-ineq-div} in~\eqref{one-step:stability-face}, we get the assertion.
\end{proof}

\subsection{Stability bounds for edge virtual elements} \label{subsection:stab-edge}
As for edge virtual elements, consider the $L^2$ projector~$\Piboldedge: \VzedgeE \rightarrow [\Pbb_0(\E)]^3$, which is computable via the degrees of freedom, i.e., the value of the tangential component over edges, as shown in~\cite{da2018lowest}.
Consider the stabilization, cf.~\cite[equation~$(4.8)$]{da2018lowest}
\begin{equation} \label{stabilization:edge}
\SEedge(\uboldh, \vboldh) := \hE^2 \sum_{\F \in \EE} \sum_{\e \in \EF} (\uboldh \cdot \te, \vboldh \cdot \te)_{0,\e} \quad \quad \forall \uboldh, \vboldh \in \VzedgeE.
\end{equation}
The stabilization~$\SEedge(\cdot, \cdot)$ is explicitly computable via the degrees of freedom since~$\uboldh \cdot \te$ and~$\vboldh \cdot \te$ belong to~$\Pbb_0(\e)$ for all edges~$\e$.

Before proving the stability result for the stabilization in~\eqref{stabilization:edge}, we show two critical preliminary results dealing with inverse estimates for functions in face and edge virtual element spaces.

\begin{lem} \label{lemma:hard-inverse-estimate}
The following inverse estimates in 2D face virtual element spaces is valid: for all faces~$\F \in \Fcaln$,
\begin{equation} \label{inverse:L2-H-1:2D-face}
\Vert \psiboldh \Vert_{0,\F} \lesssim \hF^{-1} \Vert \psiboldh   \Vert_{-1,\F}  \quad \quad \forall \psiboldh \in \Vzface(\F).
\end{equation}
Moreover, the following inverse estimates in 3D face virtual element spaces is valid as well: for all elements~$\E \in \taun$,
\begin{equation} \label{inverse:L2-H-1:3D-face}
\Vert \psiboldh \Vert_{0,\E} \lesssim \hE^{-1} \Vert \psiboldh   \Vert_{-1,\E}  \quad \quad \forall \psiboldh \in \Vzface(\E).
\end{equation}
\end{lem}
\begin{proof}
Since the proof of the two bounds is essentially identical, we focus on the first one only. Therefore, let~$\psiboldh \in \Vzface(\F)$; we recall that~$\rotF \psiboldh \in \Pbb_0(\F)$.

We start by introducing a vector polynomial~$\qboldo \in [\Pbb_1(\F)]^2$ such that~$\rotF(\psiboldh - \qboldo) = 0$,
where~$\qboldo$ is given in the local coordinates~$(x,y)$ of~$\F$ by
\[
\qboldo = \beta (\yF-y, x-\xF),\quad \quad \beta \in \Rbb.
\]
Then, we define a function~$\Psi \in H^1(\E) \setminus \Rbb$ as
\[
\begin{cases}
\DeltaF \Psi = \divF(\psiboldh-\qboldo) 					& \text{in } \F\\
\nablaF \Psi \cdot \npartialF = (\psiboldh-\qboldo) \cdot \npartialF 	& \text{on } \partial \F.
\end{cases}
\]
It is easy to check that
\begin{equation} \label{equation:sstar}
\psiboldh - \qboldo = \nablaF \Psi,
\end{equation}
since the functions on the left- and right-hand sides of~\eqref{equation:sstar} have the same divergence, rotor, and boundary normal components.
As a first step, we prove the following inverse estimate:
\begin{equation} \label{first-inverse-big-lemma}
\Vert \nablaF \Psi \Vert_{0,\F} \lesssim \hF^{-1} \Vert \nablaF \Psi \Vert_{-1,\F}.
\end{equation}
Since~$\DeltaF \Psi \in \Pbb_0(\F)$, employing a polynomial inverse estimate yields the bound
\begin{equation} \label{inverse-estimate-constantLaplacian}
\Vert \DeltaF \Psi\Vert_{0,\F} \lesssim \hF^{-1} \Vert \DeltaF \Psi \Vert_{-1,\F} = \hF^{-1} \sup_{\Phi \in H^1_0(\F)} \frac{(\DeltaF \Psi, \Phi)_{0,\F}}{\vert \Phi \vert_{1,\F}}\le \hF^{-1} \Vert \nablaF \Psi \Vert_{0,\F}.
\end{equation}
In order to show~\eqref{first-inverse-big-lemma}, we first use an integration by parts and use equation \eqref{inverse-estimate-constantLaplacian}:
\[
\Vert \nablaF \Psi \Vert_{0,\F}^2 = -\int_\F \Psi \DeltaF\Psi  + \int_{\partial \F} (\npartialF \cdot \nablaF \Psi) \, \Psi \lesssim  \hF^{-1} \Vert\Psi\Vert_{0,\F} \Vert \nablaF \Psi \Vert_{0,\F} + \Vert \npartialF \cdot \nablaF \Psi \Vert_{0,\partial \F} \Vert \Psi \Vert_{0,\partial \F}.
\]
Recall that~$\npartialF \cdot \nablaF \Psi$ is a piecewise polynomial over~$\partial \F$ and apply the polynomial inverse inequality~\eqref{L2-H1/2} on the boundary of the face~$\F$. We deduce
\[
\Vert \nablaF \Psi \Vert_{0, \F}^2 \lesssim  \hF^{-1} \Vert\Psi\Vert_{0,\F} \Vert \nablaF \Psi \Vert_{0,\F} + \hE^{-\frac{1}{2}}  \Vert \npartialF \cdot \nablaF \Psi \Vert_{-\frac{1}{2},\partial \F} \Vert \Psi \Vert_{0,\partial \F}.
\]
Next, we apply the $\div$-trace inequality~\eqref{added:1} combined with the inverse estimate~\eqref{inverse-estimate-constantLaplacian}, and arrive at
\[
\Vert \nablaF \Psi \Vert_{0, \F}^2 \lesssim  \hF^{-1} \Vert\Psi\Vert_{0,\F} \Vert \nablaF \Psi \Vert_{0,\F} + \hE^{-\frac{1}{2}}  \Vert \nablaF \Psi \Vert_{0, \F} \Vert \Psi \Vert_{0,\partial \F}.
\]
Subsequently, we use the multiplicative trace inequality~\eqref{multiplicative:trace} and get
\[
\Vert \nablaF \Psi \Vert_{0, \F} \lesssim  \hF^{-1} \Vert\Psi\Vert_{0,\F}  +  \hF^{-\frac{1}{2}} \Vert\Psi\Vert_{0,\F}^{\frac{1}{2}} \Vert \nablaF \Psi \Vert_{0,\F}^{\frac{1}{2}} \quad \quad \Longrightarrow \quad  \Vert \nablaF \Psi \Vert_{0, \F} \lesssim  \hF^{-1} \Vert\Psi\Vert_{0,\F}.
\]
The inverse inequality~\eqref{first-inverse-big-lemma} follows from the above bound and the inf-sup property of the Stokes problem, see~\cite{BBF-book}:
\[
\Vert \Psi  \Vert_{0,\F} \lesssim \sup_{\vbold \in [H^1_0(\F)]^2} \frac{(\Psi, \divF \vbold)_{0,\F}}{\vert \vbold \vert_{1,\F}} 
= \sup_{\vbold \in [H^1_0(\F)]^2} \frac{(\nablaF \Psi, \vbold)_{0,\F}}{\vert \vbold \vert_{1,\F}} = \Vert \nablaF \Psi \Vert_{-1,\F}.
\]
Recalling~\eqref{equation:sstar} and~\eqref{first-inverse-big-lemma}, we obtain
\begin{equation} \label{blue:equation}
\Vert \psiboldh - \qboldo \Vert_{0,\F} \lesssim \hF^{-1} \Vert \psiboldh - \qboldo \Vert_{-1,\F}.
\end{equation}
With this at hand, we proceed with the proof of the assertion.
We apply the triangle inequality, the inverse estimate~\eqref{blue:equation}, and a polynomial inverse estimate to obtain
\begin{equation} \label{getting-assertion-1}
\begin{split}
\Vert \psiboldh \Vert_{0,\F}
& \le \Vert \psiboldh - \qboldo \Vert_{0,\F} + \Vert \qboldo \Vert_{0,\F} \lesssim 
  \hF^{-1} \Vert \psiboldh - \qboldo \Vert_{-1,\F} + \hF^{-1} \Vert \qboldo \Vert_{-1, \F}.
\end{split}
\end{equation}
We need to show that the two terms on the right-hand side of~\eqref{getting-assertion-1} can be somewhat re-assembled together in order to obtain~\eqref{inverse:L2-H-1:2D-face}.
To this aim, we proceed as follows.
Introduce the function~$\zbold \in [H^1_0(\F)]^2$ ``realizing'' the sup in the definition of~$\Vert \psiboldh - \qboldo \Vert_{-1,\F}$. In other words, for a positive constant~$c_1$, $\zbold$ is such that
\begin{equation} \label{-1estimate-1}
(\psiboldh - \qboldo, \zbold)_{0,\F} \ge c_1 \Vert \psiboldh - \qboldo \Vert_{-1,\F} , \quad \quad \vert \zbold \vert_{1,\F}=1.
\end{equation}
Additionally, define
\[
\w \in H^2_0(\F) := \left\{ \w \in H^2(\F) \mid \w_{|\partial \F}=0,\; \nabla \w_{|\partial \F}=\mathbf 0    \right\}
\]
as the square of the piecewise cubic bubble function over the regular sub-triangulation~$\FcaltildenF$ of the face~$\F$, see the proof of Lemma~\ref{lemma:polynomial-inverse},
multiplied by~$\text{sign}(\rotF \qboldo)$.
We scale~$\w$ so that~$\vert \curlboldF \w \vert_{1,\F} = 1$
and observe that $\curlboldF \w \in [H^1_0(\F)]^2$.

Using an integration by parts, the definition of~$\w$, and several polynomial inverse estimates involving~$\qboldo$ and~$\w$, we get the following bound:
\[
\begin{split}
(\qboldo, \curlboldF \w)_{0,\F} 
& =  (\rotF \qboldo, \w)_{0,\F} = \vert \rotF \qboldo \vert \Vert \w \Vert_{L^1(\F)} \\
& \gtrsim \hF^{-2} \Vert \qboldo \Vert_{0,\F} \hF^3 \vert \curlboldF \w \vert_{1,\F} = \hF \Vert \qboldo \Vert _{0,\F} .
\end{split}
\]
An additional polynomial inverse inequality gives, for a positive constant~$c_2$,
\begin{equation} \label{-1estimate-2}
(\qboldo, \curlboldF \w)_{0,\F} \ge c_2 \Vert \qbold \Vert_{-1,\F}.
\end{equation}
We further observe the orthogonality property
\begin{equation} \label{intermediate-orthogonality}
(\psiboldh - \qboldo, \curlboldF \w)_{0,\F} = (\nabla \Psi , \curlboldF \w)_{0,\F} =0 \quad \quad \forall \w \in H^2_0(\F).
\end{equation}
Therefore, we can write
\[
\begin{split}
\Vert \psiboldh \Vert_{-1,\F}
& := \sup_{\vbold \in [H^1_0(\F)]^2} \frac{(\psiboldh, \vbold)_{0,\F}}{\vert \vbold \vert_{1,\F}} = \sup_{\vbold \in [H^1_0(\F)]^2} \frac{(\psiboldh - \qboldo, \vbold)_{0,\F} + ( \qboldo, \vbold)_{0,\F}}{\vert \vbold \vert_{1,\F}}\\
& \ge \frac{(\psiboldh - \qboldo, \zbold +\alpha \curlboldF \w )_{0,\F} + ( \qboldo, \zbold +\alpha \curlboldF \w )_{0,\F}}{\vert \zbold +\alpha \curlboldF \w \vert_{1,\F}},
\end{split}
\]
where~$\alpha>0$ is a sufficiently large constant to be assigned later. 

Using~\eqref{-1estimate-1}, \eqref{-1estimate-2}, and~\eqref{intermediate-orthogonality}, we deduce:
\begin{equation} \label{getting-assertion-2}
\begin{split}
\Vert \psiboldh \Vert_{-1,\F} 
& \ge \frac{(\psiboldh - \qboldo, \zbold)_{0,\F} + (\qboldo, \zbold)_{0,\F} + \alpha(\qboldo, \curlboldF w)_{0,\F}}{1 + \alpha \vert \curlbold w \vert_{1,\F}}\\
& \ge \frac{c_1 \Vert \psiboldh - \qboldo \Vert_{-1,\F} - \Vert \qboldo \Vert_{-1,\F} + \alpha c_2 \Vert \qboldo \Vert_{-1,\F}}{1+\alpha} \\
& = \left(  \frac{c_1}{1+\alpha}  \right) \Vert \psiboldh - \qboldo \Vert_{-1,\F} + \left(  \frac{c_2\alpha-1}{1+\alpha}   \right) \Vert \qboldo \Vert_{-1,\E}.
\end{split}
\end{equation}
The assertion follows by taking~$\alpha = 2/c_2$ and combining~\eqref{getting-assertion-1} with~\eqref{getting-assertion-2}.
\end{proof}

\begin{cor} \label{corollary:inverse-estimate-edge}
The following inverse estimates in 2D edge virtual element spaces is valid: for all faces~$\F \in \Fcaln$,
\begin{equation} \label{inverse:L2-H-1:2D-edge-local}
\Vert \vboldh \Vert_{0,\F} \lesssim \hF^{-1} \Vert \vboldh   \Vert_{-1,\F}  \quad \quad \forall \vboldh \in \Vzedge(\F).
\end{equation}
Moreover, summing over the faces of~$\E$ and using interpolation theory, we also have that
\begin{equation} \label{inverse:L2-H-1:2D-edge}
\Vert \vboldpartialEh \Vert_{0,\partial \E} \lesssim \hE^{-\frac{1}{2}} \Vert \vboldpartialEh \Vert_{-1/2, \partial \E}  \quad \quad \forall \vboldh \in \VzedgeE,
\end{equation}
where~$\vboldpartialEh$ denotes the tangential component on~$\partial \E$ of the function~$\vboldh$.
\end{cor}
\begin{proof}
The proof follows from the ``rotated'' version of~\eqref{inverse:L2-H-1:2D-face}, since in 2D the space $\Vzedge(\F)$ can be obtained by a $90^\circ$ rotation of the space~$\Vzface(\F)$, c.f. \eqref{2d-face-space} and \eqref{2d-edge-space}.
In other words, every function~$\vboldh$ in~$\Vzedge(\F)$ can be written as ${\cal R}_{\pi/2} \psiboldh$, with~$\psiboldh$ in~$\Vzface(\F)$
and the operator~${\cal R}_{\pi/2}$ denoting a rotation of~$90^\circ$ for two-component vector fields living on $F$.

We briefly comment on the proof of~\eqref{inverse:L2-H-1:2D-edge}.
Using the definition of the $H^{-1}(\F)$ norm and~\eqref{inverse:L2-H-1:2D-edge-local},  
\[
\Vert \vboldpartialEh \Vert_{0,\partial\E}
\lesssim \sum_{\F \in \EF} \Vert \vboldpartialEh{}\Vert_{0,\F}
\lesssim \sum_{\F \in \EF} \hF^{-1} \Vert \vboldpartialEh{}\Vert_{-1,\F}
\lesssim \hE^{-1} \Vert \vboldpartialEh \Vert_{-1,\partial\E}.
\]
The hidden constant depends on the number of faces of the element~$\E$.
The assertion follows by using interpolation theory.
\end{proof}

We are in the position of proving the stability result for the stabilization in~\eqref{stabilization:edge}.
\begin{prop} \label{proposition:stab-edge}
Given~$\SEedge$ the stabilization defined in~\eqref{stabilization:edge}, the following bounds are valid:
\begin{equation} \label{stability-bounds-edges}
\beta_*\Vert  \vboldh \Vert^2_{0,\E} \le   \SEedge(\vboldh, \vboldh)   \le   \beta^* \Vert  \vboldh \Vert^2_{0,\E}     \quad \quad \forall \vboldh \in \VzedgeE,
\end{equation}
where~$\beta_*$ and~$\beta^*$ are two positive constants independent of~$\hE$.
\end{prop}
\begin{proof}
The lower bound in~\eqref{stability-bounds-edges} follows directly from~\eqref{apriori-3Dedge-full}.

As for the upper bound in~\eqref{stability-bounds-edges}, we apply the polynomial inverse inequality~\eqref{inverse1D} and the 2D rot-trace inequality~\eqref{rot-trace}, and arrive at
\[
\begin{split}
 \SEedge(\vboldh, \vboldh) 
&= \hE^{2} \sum_{\F \in \EE} \sum_{\e \in \EF} \Vert  \vboldh \cdot \te   \Vert_{0, \e}^2  = \hE^{-2}\sum_{\F \in \EE} \Vert \vboldh \cdot \te \Vert^2_{0,\partial \F} \lesssim \hE \sum_{\F \in \EE} \Vert \vboldh \cdot \te \Vert^2_{-\frac{1}{2},\partial \F} \\
& \lesssim \hE \sum_{\F \in \EE} \Vert \vboldhF \Vert ^2_{0, \F} + \hE^3 \sum_{\F \in \EE} \Vert \rotF \vboldhF \Vert ^2_{0,\F}.
\end{split}
\]
As for the second term on the right-hand side we know that~$\rotF \vboldhF$ is constant over~$\F$ for all faces~$\F \in \EE$.
Therefore, using tools analogous to those employed to prove inverse estimates~\eqref{3D:inverse-curl}, \eqref{2D:inverse-Laplacian}, and~\eqref{inverse-rot}, we have
\[
\Vert \rotF \vboldhF \Vert_{0,\F} \lesssim \hE^{-1} \Vert \vboldhF \Vert_{0,\F}.
\]
Upon combining the two above bounds, we can write
\[
\SEedge(\vboldh, \vboldh)    \lesssim \hE \Vert \vboldhpartialE \Vert ^2_{0,\partial \E}.
\]
We now employ the inverse estimate~\eqref{inverse:L2-H-1:2D-edge} and the $\curl$-trace inequality~\eqref{curl-trace} to arrive at
\[
\SEedge(\vboldh, \vboldh)  \lesssim \Vert \vboldhpartialE \Vert ^2_{-\frac{1}{2},\partial \E} = \Vert \vboldh \times \nE \Vert ^2_{-\frac{1}{2},\partial \E} \lesssim \Vert \vboldh \Vert_{0,\E}^2 + \hE^{2} \Vert \curlbold \vboldh \Vert^2_{0,\E}.
\]
Recalling that~$\curlbold \vboldh \in \VzfaceE$, we apply the inverse estimate~\eqref{inverse:L2-H-1:3D-face}, integrate by parts in the dual norm definition, and deduce
\[
\begin{split}
\SEedge(\vboldh, \vboldh)   
& \lesssim \Vert \vboldh \Vert_{0,\E}^2 +  \Vert \curlbold \vboldh \Vert^2_{-1,\E} \\
& = \Vert \vboldh \Vert_{0,\E}^2 + \sup_{\Phibold \in [H^1_0(\Omega)]^2} \frac{(\curlbold \vboldh, \Phibold)_{0,\Omega}}{\vert \Phibold \vert_{1,\Omega}}  \lesssim \Vert \vboldh \Vert_{0,\E}^2  .
\end{split}
\]
\end{proof}


\begin{appendices}
\section{Interpolation and stability properties in 2D face and edge virtual element spaces} \label{appendix:interpolation-properties-2D}
Using tools analogous to those developed in Sections~\ref{section:faceVEM-loworder}--\ref{section:stabilizations} for the 3D case, we can easily derive also the analogous 2D counterparts.
For completeness, we state the interpolation properties of functions in 2D face and edge virtual element spaces.
Recall that such spaces are defined in~\eqref{2d-face-space} and~\eqref{2d-edge-space}, respectively.

\begin{thm} \label{theorem:interpolation-2D}
For any generic polygon~$\F$ of a regular 2D regular mesh in the sense of Section~\ref{section:preliminaries},
let~$\psibold$ and~$\vbold \in [H^s(\F)]^2$, $1/2 < s \le 1$. Then, there exist~$\psiboldI \in \VzfaceF$ and~$\vboldI \in \VzedgeF$ such that
\[
\Vert \psibold - \psiboldI \Vert_{0,\F} \lesssim  \hE^s \vert  \psibold \vert_{s,\F}, \quad\quad\quad \Vert \vbold - \vboldI \Vert_{0,\F} \lesssim  \hE^s \vert  \vbold \vert_{s,\F}.
\]
\end{thm}
\begin{proof}
We can prove the first bound as in Proposition~\ref{prop:best-interpolation-face}. The second bound is the ``rotated'' version of the first one.
\end{proof}
We also inherit the 2D counterpart of the stabilization properties detailed in Section~\ref{section:stabilizations}.
More precisely, for all polygons~$\F$ of the mesh, introduce the bilinear forms~$\SEface: \VzfaceF \times \VzfaceF \rightarrow \mathbb R$
\begin{equation} \label{stabilization-2D-face}
\SFface(\psiboldh, \phiboldh) : = \hE \sum_{\e \in \EF} (\psiboldh \cdot \ne, \phiboldh \cdot \ne) _{0,\e}
\end{equation}
and~$\SFedge: \VzedgeF \times \VzedgeF \rightarrow \mathbb R$
\begin{equation} \label{stabilization-2D-edge}
\SFedge(\vboldh, \wboldh) : = \hE \sum_{\e \in \EF} (\vboldh \cdot \te, \wboldh \cdot \te) _{0,\e}.
\end{equation}
\begin{thm} \label{theorem:stability-2D}
For any generic polygon~$\F$ of a regular 2D regular mesh in the sense of Section~\ref{section:preliminaries},
let~$\SFface(\cdot, \cdot)$ be the stabilization defined in~\eqref{stabilization-2D-face}. Then, the following bounds are valid:
\[
\alpha_*\Vert  \psiboldh  \Vert^2_{0,\F} \le   \SFface(\psiboldh, \psiboldh)   \le   \alpha^* \Vert  \psiboldh  \Vert^2_{0,\F}     \quad \quad \forall \psiboldh \in \VzfaceF,
\]
where~$\alpha_*$ and~$\alpha^*$ are two positive constants independent of~$\hE$.

Further, given~$\SFedge(\cdot,\cdot)$ the stabilization defined in~\eqref{stabilization-2D-edge}, the following bounds are valid:
\[
\beta_*\Vert  \vboldh \Vert^2_{0,\F} \le   \SFedge(\vboldh, \vboldh)   \le   \beta^* \Vert  \vboldh \Vert^2_{0,\F}     \quad \quad \forall \vboldh \in \VzedgeF,
\]
where~$\beta_*$ and~$\beta^*$ are two positive constants independent of~$\hE$.
\end{thm}
\begin{proof}
We can prove the first bounds as in Proposition~\ref{proposition:stab-face}. Those for edge elements are the ``rotated'' versions of the first ones.
\end{proof}
\end{appendices}

\medskip
\begin{center}
{\bf Acknowledgements}
\end{center}
\smallskip\noindent

L. Beir\~ao da Veiga was partially supported by the European Research Council through the H2020 Consolidator Grant (grant no. 681162) CAVE, ``Challenges and Advancements in Virtual Elements'', and the italian PRIN 2017 grant ``Virtual Element Methods: Analysis and Applications''. Both these supports are gratefully acknowledged.
L. Mascotto acknowledges support from the Austrian Science Fund (FWF) project P33477.
The authors are grateful to Gianni Gilardi for an interesting email exchange regarding Lemma \ref{lemma:L1-boundary-estimate}
and Remark~\ref{remark:curl-interpolation}.

{\footnotesize
\bibliography{bibliogr}
}
\bibliographystyle{plain}

\end{document}